\documentclass[a4paper,12pt]{article}
\usepackage{lscape}
\usepackage{color}
\usepackage[utf8]{inputenc}
\usepackage[english]{babel}
\usepackage{amsmath,amsfonts,amssymb,amsthm,
} 
\definecolor{darkgreen}{rgb}{0, 0.3, 0}
\newtheorem{theorem}{Theorem}[section]
\newtheorem{lemma}{Lemma}[section]

\newtheorem{proposition}[theorem]{Proposition}
\newtheorem{remark}{Remark}
\theoremstyle{plain}
\newtheorem*{theorem*}{Theorem}
\newtheorem*{conjecture*}{Conjecture}
\newtheorem*{lemma*}{Lemma}

\title{On the trace formula \\
for higher-order ODO}
\author{
E.D. Galkovskii\footnote{St. Petersburg State University; e-mail: egor\_maths@list.ru
}\setcounter{footnote}{6}
\ and
A.I. Nazarov\footnote{St. Petersburg Department of Steklov Mathematical Institute of
Russian Academy of Science
and St. Petersburg State University; e-mail: al.il.nazarov@gmail.com
}
}
\bigskip

\begin{document}

\maketitle

\begin{abstract}
A first order trace formula is obtained for a higher-order differential operator on a segment in the case where the perturbation is an operator of multiplication by a finite complex-valued measure. For the operators of even order $n\ge4$ a new term in the final formula is discovered.
\end{abstract}

\medskip

\date{}


\section*{Introduction}

We consider an operator $\mathbb L$ on a segment $[a,b]$ that is generated by a differential expression of order $n \geq 2$
\begin{align}\label{Diff_Equation_P}
{\cal L}:=(-i)^nD^n+\sum \limits_{k=0}^{n-2}{p_k(x)D^k},
\end{align}
(here $p_k\in L_1(a,b)$ are complex-valued functions) and by boundary conditions
\begin{align}\label{Boundary_Conditions}
    (P_j(D)y)(a)+(Q_j(D)y)(b) = 0,\qquad j=0, \dots, n-1.
\end{align}
Here $P_j$ and $Q_j$ are polynomials of degrees less than $n$ with complex coefficients. Denote by $d_j$ the maximum of degrees of $P_j$ and $Q_j$,
and by $a_j$ and $b_j$ the $d_j$-th coefficients of $P_j$ and $Q_j$ respectively (therefore, $a_j$, $b_j$ cannot be zeros simultaneously).

We assume that the system of boundary conditions (\ref{Boundary_Conditions}) is normalized, i.e. $\varkappa:=\sum\limits_{j=0}^{n-1}d_j$ is minimal among all the systems of boundary condition
that can be obtained from (\ref{Boundary_Conditions}) by linear bijective transformations. See \cite[ch. II, $\mathsection 4$]{Naimark} for a detailed explanation and \cite{Shk}
for a more advance treatment.

We also assume the boundary conditions (\ref{Boundary_Conditions}) to be Birkhoff regular, see \cite[ch. II, $\mathsection 4$]{Naimark}. Then the operator $\mathbb{L}$ has purely discrete
spectrum,\footnote{We underline that we do not require $\mathbb{L}$ to be self-adjoint.} which we denote by $\{\lambda_N\}_{_{N=1}}^{^\infty}$. In what follows we always enumerate the eigenvalues
in ascending order of their absolute values according to their multiplicities (that means $|\lambda_N| \leq |\lambda_{N+1}|$).

Let $\mathfrak M[a,b]$ be the space of finite complex-valued measures. Denote by $\mathbb{Q}$ the operator of multiplication by ${\mathfrak q} \in \mathfrak M[a,b]$. Then the operator
$\mathbb{L}_{\mathfrak q}=\mathbb{L}+\mathbb{Q}$ has also a purely discrete spectrum denoted by $\{\lambda_N({\mathfrak q})\}_{_{N=1}}^{^\infty}$.

    We are interested in the regularized trace
\begin{align*}
    \mathcal{S}({\mathfrak q}) := \sum_{N=1}^{\infty} \bigg[\lambda_N({\mathfrak q})-\lambda_N-\frac{1}{b-a}\int\limits_{[a,b]} {\mathfrak q}(dx)\,\bigg].
\end{align*}
Without loss of generality we suppose that $\int\limits_{[a,b]} {\mathfrak q}(dx)=0$.

    The first formula for a regularized trace was obtained by I.M. Gelfand and B.M. Levitan in 1953. In \cite{GL} they considered the problem
\begin{equation}\label{SL-GL}
    -y''+{\mathfrak q}(x)y=\lambda y; \qquad y(0) = y(\pi) = 0
\end{equation}
and showed that for a real-valued function ${\mathfrak q}(x)\in {\cal C}^1[0,\pi]$ the following relation holds:
$$
    \mathcal{S}({\mathfrak q})= -\frac{{\mathfrak q}(0)+{\mathfrak q}(\pi)}{4}.
$$
The paper \cite{GL} generated many improvements and generalizations, see a survey of V.A. Sadovnichii and V.E. Podolskii \cite{SPSurvey}.

In the recent work \cite{SZN} A.I. Nazarov, D.M. Stolyarov and P.B. Zatitskiy obtained formula
\begin{equation}\label{traceNSZ}
    \mathcal{S}({\mathfrak q}) = \frac{\psi_a(a+)}{2n}\cdot\textbf{tr}\,(\mathbb A)+\frac{\psi_b(b-)}{2n}\cdot\textbf{tr}\,(\mathbb B),
\end{equation}
for arbitrary $n\geq 2$ and regular boundary conditions, under assumptions that are standard now;\footnote{Formula (\ref{traceNSZ}) was earlier proved by R.F. Shevchenko \cite{Shv} for
a smooth function ${\mathfrak q}$ and an operator $\mathbb{L}$ without lower-order terms.} namely, ${\mathfrak q} \in L_1(a,b)$ and the functions
\begin{eqnarray}\label{psi}
\psi_a(x)=\frac{1}{x-a}\int\limits_{[a,x]} {\mathfrak q}(dt),\qquad
\psi_b(x)=\frac{1}{b-x}\int\limits_{[x,b]} {\mathfrak q}(dt)
\end{eqnarray}
have bounded variations at points $a$ and $b$ respectively. In (\ref{traceNSZ})  $\mathbb A$ and $\mathbb B$ stand for the matrices with elements that can be expressed in terms of $a_j$
and $b_j$, $j=0,\dots,n-1$. Moreover, it was shown in \cite{SZN} that in important special case, where the boundary conditions are {\bf almost separated}, the values $\textbf{tr}\,(\mathbb A)$
and $\textbf{tr}\,(\mathbb B)$ in (\ref{traceNSZ}) can be reduced and expressed using only the sums of degrees of polynomials $P_j$ and $Q_j$ respectively.

A new phenomenon was discovered in our century by A.M. Savchuk and A.A. Shkalikov \cite{S2000, SSh}, see also \cite{VinSad02}. Namely, let  ${\mathfrak q}\in \mathfrak M[0,\pi]$ be a
signed measure locally continuous at points $0$ and $\pi$. Then for the problem (\ref{SL-GL}) we have
\begin{equation}\label{SL-SSh}
    \mathcal{S}({\mathfrak q}) = -\frac{{\mathfrak q}(0)+{\mathfrak q}(\pi)}{4}-\frac{1}{8}\sum_j h_j^2,
\end{equation}
where $h_j$ stand for the jumps of the distribution function for the measure ${\mathfrak q}$. In this case the series $\mathcal{S}({\mathfrak q})$ is mean-value summable.

Thus, for ${\mathfrak q}\in \mathfrak M[a,b]$ the regularized trace becomes non-linear functional of ${\mathfrak q}$. In \cite{Konechnaya} this effect was obtained for $\delta$-potential and some other boundary conditions. See also \cite[Theorem 1]{DM} for a similar effect in a different problem.

For $n=2$ and arbitrary regular boundary conditions, the formula similar to (\ref{SL-SSh}) was obtained in \cite{NG}. Also it was shown in \cite{NG} that for $n\ge3$ a nonlinear terms does not
appear, and it was conjectured that for high order operators formula (\ref{traceNSZ}) holds for $\mathfrak q \in \mathfrak M[a,b]$.

In this paper we prove this conjecture for odd $n\geq3$ and disprove it for even $n\geq4$. Namely, in the last case formula for $\mathcal{S}({\mathfrak q})$ includes a term that has not been seen
before. This new term corresponds to the case where ${\mathfrak q}$ has an atom in the midpoint $\frac{a+b}2$.

These results were partially announced in \cite{Gal19}.
\medskip

The paper is organized as follows. In Section~\ref{Sec1} we formulate our main results (Theorems \ref{th11} and \ref{th12}) and deduce them from some interim assertions (Theorems \ref{th15} and \ref{th16}). These assertions are proved in Sections~\ref{Sec2} and \ref{Sec3} respectively. An explicit calculation of the new term is described in Section~\ref{Sec4}.\medskip

Let us recall some notation. We denote by $\mathbb{L}_0$ the operator generated by the differential expression ${\cal L}_0=(-i)^nD^n$ and boundary conditions (\ref{Boundary_Conditions}).
The eigenvalues of $\mathbb{L}_0$ are denoted by $\{\lambda_N^0\}_{_{N=1}}^{^{\infty}}$.

Further, $G_0(x,y,\lambda)$ stands for the Green function of the operator $\mathbb{L}_0-\lambda$, see \cite[ ch. I, $\mathsection 3$]{Naimark}. Notice that the resolvent $\frac 1{\mathbb{L}_0-\lambda}$
is an integral operator with the kernel $G_0(x,y,\lambda)$. So one can define the trace
$$
\mathbf{Sp}\,\frac{1}{{\mathbb L}_0 - \lambda} = \int\limits_a^b G_0(x,x,\lambda)\, dx.
$$
For arbitrary function $\Phi(\lambda)$ defined on the complex plane $\mathbb{C}$, we introduce the function $\tilde{\Phi}(z)$ by the formula
\begin{equation}\label{def_tilde}
    \tilde{\Phi}(z) = \Phi(\lambda),\quad \text{ where }\quad z = \lambda^{\frac{1}{n}}, \; Arg(z) \in [0,\frac{2\pi}{n}).
\end{equation}

   Recall the definition of summation by the mean-value method (Ces\`aro summation of order~$1$). Let $I_{\ell}$ be the sequence of partial sums corresponding to the series $\sum\limits_{j}a_j$.
The series is called mean-value summable if the following limit exists:
\begin{equation*}
({\cal C},1)\,\text{-}\lim_{\ell\to \infty} I_{\ell}:=({\cal C},1)\,\text{-}\sum \limits_{j=1}^{\infty} a_j := \lim_{k \to \infty} \frac{1}{k}\sum \limits_{\ell=1}^k I_{\ell}.
\end{equation*}

Denote by $\|\mathfrak q\|$ the total variation of $\mathfrak q$. We define the distribution function
$$
{\cal Q}(x)=\int\limits_{[a,x]} {\mathfrak q}(dt).
$$
We assume that $\mathfrak q$ has no atoms at the endpoints $a$ and $b$. This implies ${\cal Q}(a)={\cal Q}(a+)={\cal Q}(b-)={\cal Q}(b)=0$.

We say the complex-valued measure $\mathfrak q$ to be \textit{BV-regular} if the functions
$$
\psi_a(x) = \frac{{\cal Q}(x)}{x-a}, \qquad \psi_b(x) = \frac{{\cal Q}(x)}{x-b}
$$
have bounded variation on $[a,b]$. In particular, in this case the function ${\cal Q}$ is differentiable at points $a$ and $b$, and
$$
{\cal Q}'(a)=\psi_a(a+), \qquad {\cal Q}'(b)=\psi_b(b-).
$$

Let us define
$$
\Gamma^1=\Big\{w=e^{i\phi}:\phi\in\Big(0,\,\frac{\pi}{n}\Big)\Big\}; \qquad
\Gamma^2=\Big\{w=e^{i\phi}:\phi\in\Big(\frac{\pi}{n},\,\frac{2\pi}{n}\Big)\Big\}.
$$

Consider a function $R(w)$ on $\overline{\Gamma^1\cup\Gamma^2}$ such that corresponding contour
$\gamma(w) = R^n(w)w^n$ is closed and smooth. For such $R(w)$
we introduce a contour $\Gamma(w) = R(w)w$.

A sequence $\{\gamma_\ell\}$ of closed contours described above is called \textit{acceptable} if 
$R(1) = R_\ell \rightarrow \infty$ as $\ell\to\infty$, and 
for some $c_1,c_2>0$ the following conditions hold for every $\ell$:

1. $|R(w)-R(1)| \leq c_1$, $|\frac{dR(w)}{dw}| \leq c_2 R(w)$,
$w\in\overline{\Gamma^1\cup\Gamma^2}$;

2. Corresponding contours $\{\Gamma_\ell\}_{\ell=1}^\infty$ are separated from $(\lambda_N^0)^{\frac{1}{n}}$ and $(\lambda_N)^{\frac{1}{n}}$ uniformly with respect to $\ell$.

\begin{remark}\label{posled}
Assume $n$ is odd. It is well known (see, e.g., \cite[ch. II, $\mathsection 4$]{Naimark} and \cite{Shk}) that the eigenvalues are split into two sequences:
$$
\lambda_{N,j}^0 = \big((-1)^j 2 \pi N + \alpha_j+ O(\tfrac{1}{N})\big)^n,\quad j=1,2.
$$
If $n$ is even and boundary conditions (\ref{Boundary_Conditions}) are strongly regular, then (see \cite[ch. II, $\mathsection 4$]{Naimark} and \cite{Shk}) the eigenvalues are also split into two
sequences:
\begin{equation}\label{asymp}
\lambda_{N,j}^0 = \big(2 \pi N + \alpha_j+ O(\tfrac{1}{N})\big)^n,\quad j=1,2,
\end{equation}
where $\alpha_1$ and $\alpha_2$ are distinct, and $\alpha_1+\alpha_2\in\mathbb R$  (see, e.g., \cite[Theorem 1.1]{Naz09}). Therefore, in these cases there exists a sequence of acceptable contours
such that, there is exactly one eigenvalue between each two neighboring contours. Moreover, if $n$ is even and $\alpha_1 \neq \overline{\alpha}_2$ then one can take circles of radii $R_\ell^n$ as
such contours.

If $n$ is even and the boundary conditions are regular but not strongly regular then the relation (\ref{asymp}) holds with $\alpha_1=\alpha_2$. In this case we take the contours so that there is
exactly one pair of eigenvalues between each two neighboring contours.

Notice that for even $n$ the quantities $\xi_j=e^{i\alpha_j}$,
$j=1,2$, are roots of quadratic Birkhoff polynomial, see \cite[ch.
II, $\mathsection 4$]{Naimark}. Thus, we have $\xi_1\ne\xi_2$ in
strongly regular case, and $\xi_1=\xi_2$ otherwise.
\end{remark}

In what follows we use the notation $\rho=e^{i\frac {2\pi}n}$.

We denote by $\langle a\rangle$ arbitrary polynomial of $z^{-1}$ with the constant term $a$.

If the distribution function of the measure ${\mathfrak q}$ has
a jump at the point $\frac{a+b}{2}$, we denote it by $h_{\frac{a+b}{2}}$.

We introduce the function
\begin{eqnarray}\label{nu}
  \nu=\nu(w)=\begin{cases}
    \lfloor\frac{n+1}{2}\rfloor, &w \in \Gamma^1;\\
    \lfloor\frac{n}{2}\rfloor, &w \in \Gamma^2.
  \end{cases}
\end{eqnarray}

All positive constants whose exact values are not important are denoted by $C$.

\section{Formulation of the results}\label{Sec1}

Our main result consists of two following theorems:

\begin{theorem}\label{th11}
Suppose that $n \geq 3$ is odd and that the distribution function ${\cal Q}$ is differentiable at points $a$ and $b$. Then for all regular boundary conditions (\ref{Boundary_Conditions}) the
following formula holds:
\begin{equation}\label{Formula-general-odd}
    \mathcal{S}({\mathfrak q}) = \frac{{\cal Q}'(a)}{2n}\textbf{tr}\,(\mathbb A)+\frac{{\cal Q}'(b)}{2n}\textbf{tr}\,(\mathbb B).
\end{equation}

Here the matrices $\mathbb A$ and $\mathbb B$ are the same as in (\ref{traceNSZ}) (see \cite[Theorem 2]{SZN}).
The series for $\mathcal{S}({\mathfrak q})$ converges in a usual way.
\end{theorem}

For $n$ even, a new term appears. It depends on the value of a jump of the distribution function at the point $\frac{a+b}{2}$. 

\begin{theorem}\label{th12}
Suppose that $n \geq 4$ is even and that the complex-valued measure ${\mathfrak q}\in \mathfrak M[a,b]$ is BV-regular. Then for all regular boundary conditions (\ref{Boundary_Conditions}) the following
formula holds:
\begin{equation}\label{Formula-general}
    \mathcal{S}({\mathfrak q}) = \frac{{\cal Q}'(a)}{2n}\textbf{tr}\,(\mathbb A)+\frac{{\cal Q}'(b)}{2n}\textbf{tr}\,(\mathbb B)+\frac{h_{\frac{a+b}{2}}}{2\pi}\,{\mathfrak{C}}.
\end{equation}
Here the matrices $\mathbb A$ and $\mathbb B$ are the same as in (\ref{traceNSZ}) (see \cite[Theorem 2]{SZN}), and the coefficient $\mathfrak{C}$ is defined in (\ref{mathfrakC}).

If the boundary conditions (\ref{Boundary_Conditions}) are
strongly regular (this corresponds to distinct roots of the
Birkhoff polynomial, $\xi_1\ne\xi_2$)
 then the series for $\mathcal{S}({\mathfrak q})$ can be summed by the Ces\'aro method.

If the boundary conditions (\ref{Boundary_Conditions}) are regular but not strongly regular (this corresponds to the case $\xi_1=\xi_2$) then the series can be summed by the Ces\'aro method with brackets. Namely, the terms related to coinciding or asymptotically close eigenvalues
are summed pairwise, and then the appeared series is summed by the Ces\'aro method.
\end{theorem}

The value of the constant $\mathfrak{C}$ is given by the following theorem.

\begin{theorem}\label{th13}
   Suppose that $n \geq 4$ be even and that a sequence of acceptable contours $\gamma_\ell$ is chosen in accordance to Remark \ref{posled}.

If the boundary conditions (\ref{Boundary_Conditions}) are strongly regular (recall that this corresponds to the case $\xi_1\ne\xi_2$), then
\begin{equation}\label{th13-log}
    \mathfrak{C} = \mathfrak c\,\frac{Log(-\xi_2/\xi_1)}{\xi_1-\xi_2}.
\end{equation}
 
 If the boundary conditions (\ref{Boundary_Conditions}) are regular, but not strongly regular (this corresponds to the case $\xi_1=\xi_2$), then
\begin{equation}\label{th13-nolog}
    \mathfrak{C} = -\frac{\mathfrak c}{\xi_1}=-\frac{\mathfrak c}{\xi_2}.
\end{equation}
The constant $\mathfrak c$ is defined in (\ref{frak-c}) and depends only on the leading coefficients of polynomials $P_j$ and $Q_j$ in boundary conditions (\ref{Boundary_Conditions}).
The $Log$ sign in (\ref{th13-log}) stands for the branch of logarithm with $\Im Log \in (-\pi, \pi)$.
\end{theorem}

\begin{remark}
 In the case $\xi_2/\xi_1\in \mathbb R_+$ we cannot define the natural order of the eigenvalues since $\alpha_1=\overline\alpha_2$ in (\ref{asymp}). In this case the choice
$\Im Log=\pm\pi$ in (\ref{th13-log}) depends on the order of summation.

 Notice also that formula (\ref{th13-nolog}) differs from the limit of (\ref{th13-log}) as $\xi_2\to\xi_1$, since the series for $\mathcal{S}({\mathfrak q})$ is summed in different ways.
\end{remark}

To prove Theorem \ref{th11} and Theorem \ref{th12} we need the
following statement.

\begin{proposition}[\cite{NG}, Theorem 2.2]
\label{prop14}
    Let $n \geq 3$. For every acceptable sequence of contours $\gamma_\ell$ the following relation holds as $\ell \rightarrow \infty$ (summation in the left hand side is taken over
$\lambda_N({\mathfrak q})$, $\lambda_N$ that are inside $\gamma_\ell$):
    \begin{equation}\label{prop_equality}
\sum\Big[\lambda_N({\mathfrak q})-\lambda_N\Big]  = \frac{i}{2 \pi}\int\limits_{\gamma_\ell}\int\limits_{[a,b]} G_0(x,x,\lambda)\,{\mathfrak q}(dx)\,d\lambda + o(1).
    \end{equation}
\end{proposition}

Passage to the limit in the right hand side of (\ref{prop_equality}) is provided by the following interim statements.

\begin{theorem}\label{th15}
   Let $n \geq 3$ be odd, and let ${\cal Q}'(a) = {\cal Q}'(b) = 0$. Then for every sequence of acceptable contours $\gamma_\ell$ the following equality holds:
\begin{equation*}
    \lim_{\ell\to \infty}\int\limits_{\gamma_\ell}\int\limits_{[a,b]} G_0(x,x,\lambda){\mathfrak q}(dx)\,d\lambda = 0.
\end{equation*}
\end{theorem}

\begin{theorem}\label{th16}
   Let $n \geq 4$ be even, and let a complex-valued measure ${\mathfrak q}\in \mathfrak M[a,b]$ be BV-regular. Assume also that ${\cal Q}'(a) = {\cal Q}'(b) = 0$.
Finally, let a sequence of acceptable contours $\gamma_\ell$ be chosen in accordance to Remark~\ref{posled}. Then the following equality holds:
\begin{equation}\label{C1=0}
    ({\cal C},1)\,\text{-}\lim_{\ell\to \infty}\int\limits_{\gamma_\ell}\int\limits_{[a,b]} G_0(x,x,\lambda){\mathfrak q}(dx)\,d\lambda
    = -i \mathfrak{C} \cdot h_{\frac{a+b}{2}}.
\end{equation}
\end{theorem}

\begin{proof}[Proof of Theorems \ref{th11} and \ref{th12}]
We decompose the measure $\mathfrak q$ into two parts:
\begin{align*}
\mathfrak q = \mathfrak{q}_0 + \mathfrak{q}_1,
\end{align*}
where $\mathfrak{q}_0$ is a smooth function with $\mathfrak{q}_0(a)={\cal Q}'(a)$, $\mathfrak{q}_0(b)={\cal Q}'(b)$, and $\int\limits_{[a,b]} {\mathfrak q}_0(dx)=0$.

Let $n\ge3$ be odd.
Then $\mathfrak{q}_1$ satisfies the assumptions of Theorem \ref{th15}. Since
$\mathcal{S}({\mathfrak q})=\mathcal{S}({\mathfrak q}_0)+\mathcal{S}({\mathfrak q}_1)$, formula (\ref{Formula-general-odd}) follows from (\ref{traceNSZ}) for ${\mathfrak q}_0$, (\ref{prop_equality})
and Theorem \ref{th15} for ${\mathfrak q}_1$.

Now let $n\ge4$  be even. Then $\mathfrak{q}_1$ satisfies the assumptions of Theorem \ref{th16}. Formula (\ref{Formula-general}) follows from (\ref{traceNSZ}) for ${\mathfrak q}_0$,
(\ref{prop_equality}) and Theorem \ref{th16} for ${\mathfrak q}_1$.
\end{proof}

\section{Auxiliary estimates. Proof of Theorem~\ref{th15}}\label{Sec2}

Here and further we assume without loss of generality that $a = 0$, $b = 1$.

We begin with the explicit formula for the Green function, see \cite[formula (12)]{SZN}. For $y=x$ this gives:

\begin{align}\label{n-square}
\tilde{G}_0(x,x,z)=-\frac{i}{n z^{n-1}}\sum\limits_{\alpha,\beta=1}^n \rho^{\alpha-1}e^{i z x(\rho^{\beta-1}-\rho^{\alpha-1})}\cdot\frac{\Delta_{\alpha,\beta}(z)}{\Delta(z)}
\end{align}
(the determinants $\Delta(z)$, $\Delta_{\alpha,\beta}(z)$ are introduced in Appendix).

For the sake of brevity introduce a notation $k = \frac{n}{2}$ for even $n$.

\begin{lemma}\label{lemma31}
Let a pair $(\alpha,\beta)$, $\alpha\neq\beta$, be arbitrary for odd $n$, and let
\begin{equation}\label{alphabeta}
    (\alpha,\beta)\ne(1,k+1), \ (k+1,1), \ (k,n), \ (n,k)
\end{equation}
for even $n$. Then for every acceptable
sequence of contours $\gamma_\ell$ the following estimate holds:
\begin{align*}
\left|e^{iR(w)w x(\rho^{\beta-1}-\rho^{\alpha-1})}\frac{\Delta_{\alpha,\beta}(R(w) w)}{\Delta(R(w) w)}\right| \leq C e^{-c_0 R(w) \min(x,1-x)},
\end{align*}
where $C$, $c_0 > 0$.
\end{lemma}

\begin{proof}
By Proposition \ref{proposition21} we obtain
\begin{align*}
\left|e^{iR(w)w x(\rho^{\beta-1}-\rho^{\alpha-1})}\frac{\Delta_{\alpha,\beta}(R(w) w)}{\Delta(R(w) w)}\right| \leq C e^{R(w)\Psi_{\alpha,\beta}(w,x)},
\end{align*}
where
$$
\Psi_{\alpha,\beta}(w,x) = iw\rho^{\alpha-1}\eta^1_{\alpha}(w, x)+iw\rho^{\beta-1}\eta^2_{\beta}(w, x);
$$
\begin{eqnarray*}
    \eta^1_{\alpha}(w, x) =
\left\{
    \begin{array}{lr}
        1-x, & \alpha \leq \nu(w), \\
        0-x, & \alpha > \nu(w); \\
    \end{array}
\right.\quad
\eta^2_{\beta}(w, x) =
\left\{
    \begin{array}{lr}
        x-0, & \beta \leq \nu(w), \\
        x-1, & \beta > \nu(w). \\
    \end{array}
\right.
\end{eqnarray*}

We recall that the function $\nu(w)$ is introduced in (\ref{nu}) and differs for odd and even $n$. Namely, If $n$ is even, then $\nu(w)=\frac{n}{2}$ for all $w\in\Gamma^1\cup\Gamma^2$. If $n$ is odd, then $\nu(w)=\frac{n+1}{2}$ for $w\in\Gamma^1$
and $\nu(w)=\frac{n-1}{2}$ for $w\in\Gamma^2$.

Notice that the real part of both summands in $\Psi_{\alpha,\beta}(w,x)$ is negative for all $(w,x) \in (\Gamma^1\cup\Gamma^2)\times(0,1)$. Moreover, if $n$ is odd then the real part of
$iw\rho^{m-1}$, $w\in\overline{\Gamma^1\cup\Gamma^2}$, can be equal to zero only in three cases:
\begin{itemize}
\item $m = 1$, $Arg(w) = 0$,
\item $m = \frac{n-1}{2}$, $Arg(w) = \frac{\pi}{n}$,
\item $m = n$, $Arg(w) = \frac{2\pi}{n}$.
\end{itemize}
Therefore, for $n$ odd, the quantity $|\Re(iw\rho^{\alpha-1})|+|\Re(iw\rho^{\beta-1})|$ is separated from zero for all $\alpha \neq \beta$ uniformly w.r.t. $w\in\overline{\Gamma^1\cup\Gamma^2}$.
This implies
\begin{eqnarray*}
&& \Re(\Psi_{\alpha,\beta}(w,x)) 
=-|\Re(iw\rho^{\alpha-1})|\cdot|\eta^1_{\alpha}(x)|-|\Re(iw\rho^{\beta-1})|\cdot|\eta^2_{\beta}(x)| \\
&\leq& -\min(x,1-x)(|\Re(iw\rho^{\alpha-1})|+|\Re(iw\rho^{\beta-1})|) \leq -c_0\min(x,1-x)
\end{eqnarray*}
for some positive constant $c_0$. This proves Lemma for odd $n$.\medskip

If $n$ is even then $\Re(iw\rho^{m-1})$ can be equal to zero only in four cases:
\begin{itemize}
\item $m = 1$, $Arg(w) = 0$,
\item $m = k$, $Arg(w) = \frac{2\pi}{n}$,
\item $m = k+1$, $Arg(w) = 0$,
\item $m = n$, $Arg(w) = \frac{2\pi}{n}$.
\end{itemize}

In all other cases $|\Re(i w \rho^{m-1})|$ is separated from zero uniformly w.r.t. $w \in \overline{\Gamma^1\cup\Gamma^2}$. Thus, the inequality $\Re(\Psi_{\alpha,\beta}(w,x))\leq -c_0\min(x,1-x)$
holds for even $n$ provided $(\alpha,\beta)$ satisfy (\ref{alphabeta}). This proves Lemma for even $n$.
\end{proof}

\begin{lemma}\label{lemma32}
Let the distribution function of a complex-valued measure $\mathfrak q$ satisfy ${\cal Q}'(0) = {\cal Q}'(1) = 0$.\footnote{In this Lemma we do not assume that $\int\limits_{[0,1]} {\mathfrak q}(dx)=0$.} Suppose that $\Xi$ is a bounded function, and
$$
\Psi\in {\cal C}^1(\overline{\Gamma^1 \cup \Gamma^2}\times[0,1]); \qquad \Re(\Psi(w,x)) \leq -c_0 \min(x,1-x).
$$
Then the following relation holds for $R=R(w) \rightrightarrows \infty$:
$$
\int\limits_{\Gamma^1 \cup \Gamma^2}\int\limits_{[0,1]} R e^{R \Psi(w,x)}\Xi(Rw)\mathfrak q(dx)dw = o(1).
$$
\end{lemma}

\begin{proof}
We choose a point $\widehat x\in(0,1)$ such that $\cal Q$ is continuous at $\widehat x$ and split the integration segment $[0,1]$ into two parts: $x \in [0,\widehat x]$ and $x \in [\widehat x,1]$. We prove the estimate of the first integral, the second one is estimated similarly.

Integration by parts w.r.t. $x$ gives
\begin{eqnarray*}
\int\limits_{[0,\widehat x]} R e^{R \Psi(w,x)}\Xi(Rw)\mathfrak q(dx) &=& R {\cal Q}(\widehat x)e^{R \Psi(w,\widehat x)}\Xi(Rw)\\
&-& \int\limits_0^{\widehat x} R^2\Psi'_x(w,x) {\cal Q}(x)e^{R \Psi(w,x)}\Xi(Rw)dx.
\end{eqnarray*}
The first term here is $O(Re^{-c_0R\widehat x})$ uniformly in $w$.
To manage the second term we define $\tau_R = R(1)^{-\frac{1}{2}}$. By assumption ${\cal Q}'(0) =0$ we have $|{\cal Q}(x)| \leq \varepsilon_R x$ for $x\in[0,\tau_R]$ where $\varepsilon_R \rightarrow 0$ as $R \rightarrow \infty$.
Therefore
we have
\begin{eqnarray*}
&&\Big|\,\int\limits_{\Gamma^1 \cup \Gamma^2} \int\limits_0^{\widehat x} R^2\Psi'_x(w,x) {\cal Q}(x)e^{R \Psi(w,x)}\Xi(Rw)dxdw\Big|\\
&\leq& C \varepsilon_R\Big|\int\limits_0^{\tau_R} R^2 x e^{-c_0 R x}dx\Big| + C \|\mathfrak q\| R^2 e^{-c_0R\tau_R} \le C \varepsilon_R+CR^2 e^{-c_0\sqrt{R}},
\end{eqnarray*}
and the statement follows.
\end{proof}

\begin{proof}[Proof of Theorem \ref{th15}]
We rewrite the integral using the representation (\ref{n-square}):
\begin{eqnarray}\nonumber
    &&\int\limits_{\gamma_\ell}\int\limits_{[0,1]} G_0(x,x,\lambda)\mathfrak q(dx)d\lambda \\
    \label{int-n-sq}
    &=&\int\limits_{\Gamma_\ell}\int\limits_{[0,1]} \tilde{G}_0(x,x,z)nz^{n-1}\mathfrak q(dx)dz = -i\sum_{\alpha,\beta=1}^n \,\int\limits_{\Gamma_\ell}{\cal I}_{\alpha,\beta}(z)\,dz,
\end{eqnarray}
where
\begin{eqnarray}\label{int-alpha-beta}
    {\cal I}_{\alpha,\beta}(z)=\int\limits_{[0,1]} \rho^{\alpha-1}e^{i z x (\rho^{\beta-1}-\rho^{\alpha-1})}\cdot\frac{\Delta_{\alpha,\beta}(z)}{\Delta(z)}\,\mathfrak q(dx).
\end{eqnarray}
If $\alpha = \beta$, the integral (\ref{int-alpha-beta}) equals zero by the assumption $\int\limits_{[0,1]} {\mathfrak q}(dx)=0$. For $\alpha\ne\beta$, we write
\begin{eqnarray*}
\int\limits_{\Gamma_\ell}{\cal I}_{\alpha,\beta}(z)\,dz 
&=& \int\limits_{\Gamma^1 \cup \Gamma^2}\int\limits_{[0,1]}
(R(w)+R'(w)w)\\
&\times& \rho^{\alpha-1}e^{i R(w)w x (\rho^{\beta-1}-\rho^{\alpha-1})}\cdot\frac{\Delta_{\alpha,\beta}(R(w)w)}{\Delta(R(w)w)}\,\mathfrak q(dx)dw.
\end{eqnarray*}
Lemma \ref{lemma31} and the property 1 of admissible contours give the estimate of integrand which allows to apply Lemma
\ref{lemma32}. Therefore, the integral tends to zero
as $\ell \rightarrow \infty$, and the statement follows.
\end{proof}

\section{Proof of theorem \ref{th16}}\label{Sec3}

The starting point of the proof is the same as in Theorem \ref{th15}. We use decomposition (\ref{int-n-sq}), (\ref{int-alpha-beta}). The terms with $\alpha = \beta$ vanish by the assumption
$\int\limits_{[0,1]} {\mathfrak q}(dx)=0$. Then, using Lemmata \ref{lemma31} and \ref{lemma32} we obtain for $n=2k$
\begin{align}
\nonumber
&\int\limits_{\gamma_\ell}\int\limits_{[0,1]} G_0(x,x,\lambda)\mathfrak q(dx)d\lambda\\
= -i &\int\limits_{\Gamma_\ell}\big({\cal I}_{1,k+1}+{\cal I}_{k+1,1}+{\cal I}_{k, n}+{\cal I}_{n, k}\big)\,dz+o(1)
\label{4int}
\end{align}
as $\ell\to\infty$, where
\begin{align*}
{\cal I}_{1,k+1} &= {\cal I}_{1,k+1}(z) = \int\limits_{[0,1]} e^{2iz(1-x)} \cdot\frac{\hat{\Delta}_{1,k+1}(z)}{\hat{\Delta}(z)}\, \mathfrak q(dx);\\
{\cal I}_{k+1,1} &= {\cal I}_{k+1,1}(z) = \int\limits_{[0,1]} \rho^{k}e^{2izx} \cdot\frac{\hat{\Delta}_{k+1,1}(z)}{\hat{\Delta}(z)}\, \mathfrak q(dx);
\end{align*}
\begin{align*}
{\cal I}_{k, n} &= {\cal I}_{k, n}(z) = \int\limits_{[0,1]} \rho^{k-1}e^{2iz(1-x)\rho^{k-1}} \cdot\frac{\hat{\Delta}_{k,n}(z)}{\hat{\Delta}(z)}\, \mathfrak q(dx);\\
{\cal I}_{n, k} &= {\cal I}_{n, k}(z) = \int\limits_{[0,1]} \rho^{n-1} e^{2izx\rho^{k-1}} \cdot\frac{\hat{\Delta}_{n, k}(z)}{\hat{\Delta}(z)}\, \mathfrak q(dx).
\end{align*}
Generally speaking, these four terms do not converge in the usual sense as $\ell\to\infty$, and we use the $({\cal C},1)$-passage to the limit.

We split the measure $\mathfrak q(dx)$ as follows:
\begin{align*}
    {\mathfrak q}(dx)=h_{\frac{1}{2}}\cdot\delta\big(x-\frac 12\big)+{\mathfrak q}_1(dx),
\end{align*}
so that ${\mathfrak q}_1(dx)$ has no atom at the point $\frac 12$ but in general $\int\limits_{[0,1]} {\mathfrak q}_1(dx)\ne0$.

Therefore,
\begin{align}
\nonumber
&({\cal C},1)\,\text{-}\lim\limits_{\ell\rightarrow\infty}\int\limits_{\gamma_\ell}\int\limits_{[0,1]} G_0(x,x,\lambda)\mathfrak q(dx)d\lambda
=-i {\mathfrak C}\cdot h_{\frac{1}{2}}\\
-i \cdot&({\cal C},1)\,\text{-}\lim\limits_{\ell\rightarrow\infty}\int\limits_{\Gamma_\ell}\big({\cal I}^1_{1,k+1}+{\cal I}^1_{k+1,1}+{\cal I}^1_{k, n}+{\cal I}^1_{n, k}\big)\,dz.
\label{4int+1}
\end{align}
Here we denote by ${\cal I}^1_{\alpha,\beta}$ the integrals similar to ${\cal I}_{\alpha,\beta}$ with ${\mathfrak q}_1$ instead of $\mathfrak q$, while
\begin{align}
\label{mathfrakC}
    {\mathfrak C} &= ({\cal C},1)\,\text{-}\lim\limits_{\ell\rightarrow\infty} \int\limits_{\Gamma_\ell}\big({\mathbb I}^{\delta}_{1,k+1}(z)+{\mathbb I}^{\delta}_{k+1,1}(z)
    +{\mathbb I}^{\delta}_{k, n}(z)+{\mathbb I}^{\delta}_{n, k}(z)\big)\,dz;
\end{align}
\begin{equation}\label{mathbb-I}
\aligned
{\mathbb I}^{\delta}_{1,k+1}(z) &= e^{iz} \cdot\frac{\hat{\Delta}_{1,k+1}(z)}{\hat{\Delta}(z)};
\qquad {\mathbb I}^{\delta}_{k, n}(z) = \rho^{k-1}e^{iz\rho^{k-1}} \cdot\frac{\hat{\Delta}_{k,n}(z)}{\hat{\Delta}(z)}; \\
{\mathbb I}^{\delta}_{k+1,1}(z) &= \rho^{k}e^{iz} \cdot\frac{\hat{\Delta}_{k+1,1}(z)}{\hat{\Delta}(z)}; \quad {\mathbb I}^{\delta}_{n, k}(z) = \rho^{n-1} e^{iz\rho^{k-1}}
\cdot\frac{\hat{\Delta}_{n, k}(z)}{\hat{\Delta}(z)}.
\endaligned
\end{equation}
The first term in (\ref{4int+1}) gives us the right-hand side in formula (\ref{C1=0}). Thus, we should demonstrate that the second Ces\'aro limit in (\ref{4int+1}) equals zero.
We proceed in two steps.

On the first step, we take a sequence of acceptable contours $\gamma_\ell$ such that, there is exactly one pair of eigenvalues between each two neighboring contours. Notice that it is always possible
to choose circles of radii $R_\ell^n$ as such contours (cf. Remark \ref{posled}).

\begin{lemma}\label{lemma41}
Let the complex-valued measure $\mathfrak{q}$ satisfy the assumptions of Theorem {\ref{th16}}. Consider the decomposition (\ref{4int+1}). Then for arbitrary sequence $R_\ell\rightarrow\infty$ such that circles of radii $R_\ell^n$ separate pairs
of eigenvalues, the following relation holds:
\begin{align*}
    ({\cal C},1)\,\text{-}\lim\limits_{\ell\rightarrow\infty} \int\limits_{R_\ell (\Gamma^1\cup\Gamma^2)}\!\!\big({\cal I}^1_{1,k+1}+{\cal I}^1_{k+1,1}+{\cal I}^1_{k, n}+{\cal I}^1_{n, k}\big)\,dz = 0.
\end{align*}
\end{lemma}

\begin{proof}
We prove that for the integral of ${\cal I}^1_{k+1,1}$, $({\cal C},1)$-limit equals zero. For other terms, the proof is quite similar. By (\ref{asymp}), we can assume without loss of generality that
$R_\ell=R_0+2\ell\pi$ for large $\ell$.

 Similarly to Lemma \ref{lemma32}, the integral over $R_\ell \Gamma^2$ tends to zero. Next, using the Cauchy residue theorem we replace the integral over the arc $R_\ell \Gamma^1$
by the integral over two segments (see Fig. \ref{fig:1})
$$
(R_\ell,0)\rightarrow(R_\ell,R_\ell\sin(\tfrac{\pi}{n}))\rightarrow(R_\ell \cos(\tfrac{\pi}{n}),R_\ell \sin(\tfrac{\pi}{n})).
$$

\begin{figure}
\setlength{\unitlength}{1cm}
\begin{picture}(0,6)(-1, -1)

\put(0, 0){\vector(1, 0){10}}
\put(0, 0){\vector(0,1){7}}
\qbezier(0,0)(7.071068,7.071068)(7.071068,7.071068)

\thicklines
\put(-0.500000,-0.500000){0}
\put(9.000000,-0.500000){$Re(z)$}
\put(-1.200000,6.656854){$Im(z)$}
\put(8.050000,0.200000){$(R, 0)$}
\put(8.050000,5.656854){$(R, R\sin(\frac{\pi}{n}))$}
\put(2.056854,5.656854){$(R \cos(\frac{\pi}{n}), R \sin(\frac{\pi}{n}))$}
\qbezier(8.000000,0)(8.000000,3.061467)(5.656854,5.656854)
{\color{red}
\qbezier(5.656854,5.656854)(8.000000,5.656854)(8.000000,5.656854)
\qbezier(8.000000,5.656854)(8.000000,5.656854)(8.000000,0.000000)}
\end{picture}
    \caption{}
    \label{fig:1}
\end{figure}

Since $R_\ell$ are separated from $|\lambda_N^0|^{\frac{1}{n}}$, the new contours are separated from $|\lambda_N^0|^{\frac{1}{n}}$ for large $\ell$. Similarly to the proof of Theorem \ref{th15}, 
we show using Lemmata \ref{lemma31} and \ref{lemma32} that the integral over the second segment also tends to zero. This gives
\begin{align}
\nonumber
&\int\limits_{R_\ell (\Gamma^1\cup\Gamma^2)}\!\!{\cal I}^1_{k+1,1}(z)\,dz\\
= -i &\int\limits_0^{R_\ell \sin(\frac{\pi}{n})}\int\limits_{[0,1]} e^{2i(R_\ell+i\tau)x} \cdot\frac{\hat{\Delta}_{k+1,1}(R_\ell+i\tau)}{\hat{\Delta}(R_\ell+i\tau)}\, \mathfrak{q}_1(dx)d\tau + o(1).
\label{I_k+1_1}
\end{align}

One can see from (\ref{frak-m}) that if $w=e^{iArg(z)}\in\Gamma^1$ then
$\hat\Delta(z)$ is a polynomial of variables $e^{iz}$ and $z^{-1}$, and its degree with respect to $e^{iz}$ equals two. 
So, if $z=R_\ell+i\tau$, $\tau\in(0,R_\ell \sin(\frac{\pi}{n}))$ then
\begin{align*}
    \hat\Delta(z) = M_1(e^{iz}) + z^{-1}M_2(e^{iz})+O(z^{-2}) \quad \mbox{as} \quad \ell\to\infty,
\end{align*}
where $M_1$ and $M_2$ are polynomials of degree two. The constant term $\mathfrak m$ in the polynomial $M_1$ does not equal zero, because conditions (\ref{Boundary_Conditions}) are Birkhoff regular.

We decompose $\hat{\Delta}_{k+1,1}(z)$ in a similar way and obtain the following relation as $\ell\to\infty$:
\begin{align}\label{sumfour}
\frac{\hat{\Delta}_{k+1,1}(z)}{\hat{\Delta}(z)} = \frac{\mathfrak{m}_{k+1,1}}{\mathfrak{m}} + e^{iz}\mathbb{M}_1(e^{iz}) + \frac{C}{z} + \frac{e^{iz}}{z}\mathbb{M}_2(e^{iz}) + O(z^{-2}),
\end{align}
where $\mathbb{M}_1$ and $\mathbb{M}_2$ are proper rational functions, and their denominators are polynomials with non-zero constant terms.

We split the integral (\ref{I_k+1_1}) into a sum corresponding to decomposition (\ref{sumfour}) and estimate these integrals one by one:
\begin{align*}
\int\limits_{R_\ell (\Gamma^1\cup\Gamma^2)}\!\!{\cal I}^1_{k+1,1}(z)\,dz = J_1(\ell) + J_2(\ell) + J_3(\ell) + J_4(\ell) + J_5(\ell).
\end{align*}
It is evident that $J_5(\ell) = o(1)$ as $\ell \rightarrow \infty$.

We start with the fourth term:
\begin{eqnarray*}
|J_4(\ell)| = \bigg|\int\limits_0^{R_\ell \sin(\frac{\pi}{n})}\!\int\limits_{[0,1]} e^{2i(R_\ell+i\tau)x}\frac{e^{iR_\ell-\tau}}{R_\ell +i\tau}\mathbb{M}_2\left(e^{iR_\ell-\tau}\right)
\mathfrak{q}_1(dx)d\tau\bigg| \\
\leq C\cdot\bigg|\!\int\limits_0^{R_\ell \sin(\frac{\pi}{n})}\!\!\|\mathfrak{q}_1\|\, \frac{e^{-\tau}}{R_\ell}\, d\tau\bigg| = o(1)\quad \mbox{as} \quad \ell\to\infty.
\end{eqnarray*}

To estimate the third integral
\begin{align*}
i J_3(\ell) = \int\limits_0^{R_\ell \sin(\frac{\pi}{n})}\! \int\limits_{[0,1]} e^{2i(R_\ell+i\tau)x} \frac{C}{R_\ell+i\tau}\, \mathfrak{q}_1(dx)d\tau,
\end{align*}
we observe that the function
\begin{align*}
\int\limits_0^{R_\ell \sin(\frac{\pi}{n})}  e^{2iR_\ell x} \frac{e^{-2\tau x}}{R_\ell+i\tau}\, d\tau
\end{align*}
is uniformly bounded and converges to zero as $\ell\to\infty$ for all $x\in(0,1]$. Since ${\mathfrak q}_1$ has no atom at zero, the integral tends to zero by the Lebesgue dominated convergence theorem.

Next, we transform the second integral as follows:
\begin{align*}
iJ_2(\ell)= \int\limits_0^{R_\ell\sin(\frac{\pi}{n})}\!\int\limits_{[0,1]} e^{2i(R_\ell+i\tau)(1/2+x)} \mathbb{M}_1\left(e^{iR_\ell-\tau}\right)\, \mathfrak{q}_1(dx)d\tau \\
=\int\limits_{[0,1]} F(x)e^{2iR_\ell(1/2+x)} \mathfrak{q}_1(dx)+o(1),
\end{align*}
where
$$
F(x) = \int\limits_0^{\infty} \mathbb{M}_1(e^{i(R_0 + i\tau)})e^{-2\tau(1/2+x)}d\tau,
$$
because $\mathbb{M}_1\left(e^{iR_\ell-\tau}\right)=\mathbb{M}_1\left(e^{iR_0-\tau}\right)$ (recall that $R_\ell=R_0+2\ell\pi$) and
\begin{align*}
\bigg|\int\limits_{R_\ell \sin(\frac{\pi}{n})}^{\infty}\int\limits_{[0,1]}\! e^{2i(R_\ell+i\tau)(1/2+x)} \mathbb{M}_1(e^{i(R_0+i\tau)}) \mathfrak{q}_1(dx)d\tau\bigg| \leq \!
\int\limits_{R_\ell \sin(\frac{\pi}{n})}^{\infty}\!\!\!\|\mathfrak{q}_1\| e^{-\tau} d\tau = o(1).
\end{align*}

Now we are in position to apply the Ces\'aro method:
\begin{align*}
&({\cal C},1)\,\text{-}\lim_{\ell\rightarrow\infty}\int\limits_{[0,1]} F(x)e^{2iR_\ell(1/2+x)} \mathfrak{q}_1(dx)\\
=& \lim\limits_{\ell\rightarrow\infty}\int\limits_{[0,1]} F(x)\,\frac{e^{2 i R_0(1/2+x)}}{\ell}\frac{1-e^{4\pi i \ell (1/2+x)}}{1-e^{4\pi i (1/2+x)}}\,\mathfrak{q}_1(dx).
\end{align*}
 It is easy to see that $F(x)$ is a continuous bounded function. Therefore, the last integrand is uniformly bounded and converges to zero as $\ell\to\infty$ for all $x\not\in\{0, \frac 12, 1\}$.
 Since $\mathfrak{q}_1$ has no atoms at these points, the integral tends to zero by the Lebesgue theorem.

To deal with the remaining integral $J_1$, we recall that $\mathfrak{q}_1$ is BV-regular. Thus, $\mathfrak{q}_1(dx) = \psi_0(x)dx+x \cdot d\psi_0(x)$ and $\psi_0$ has bounded variation at zero.
This gives
\begin{align*}
iJ_1(\ell)=& \int\limits_0^{R_\ell \sin(\frac{\pi}{n})}\int\limits_0^1 \psi_0(x) e^{2 i R_\ell x - 2\tau x }
\frac{\mathfrak{m}_{k+1,1}}{\mathfrak{m}} \,dxd\tau +\,\frac{1}{2} \int\limits_0^1 e^{2 i R_\ell x}\frac{\mathfrak{m}_{k+1,1}}{\mathfrak{m}} \,d\psi_0(x)\\
-& \,\frac{1}{2} \int\limits_0^1 e^{2 R_\ell x(i-\sin(\frac{\pi}{n}))}\frac{\mathfrak{m}_{k+1,1}}{\mathfrak{m}} \,d\psi_0(x)=: J_{11}(\ell)+J_{12}(\ell)+J_{13}(\ell).
\end{align*}
We integrate $J_{11}(\ell)$ by parts with respect to $x$. The boundary term at $0$ vanishes due to $\psi_0(0+)={\cal Q}'(0)=0$, and we obtain
\begin{align}
\nonumber
J_{11}(\ell)=&- \int\limits_0^{R_\ell \sin(\frac{\pi}{n})}\int\limits_0^1 e^{2 i R_\ell x}\frac{e^{- 2\tau x }}{2 i R_\ell - 2 \tau}
\frac{\mathfrak{m}_{k+1,1}}{\mathfrak{m}} \,d\psi_0(x)d\tau \\+ 
& \int\limits_0^{R_\ell \sin(\frac{\pi}{n})} e^{2 i R_\ell }\frac{\psi_0(1)e^{- 2\tau }}{2 i R_\ell - 2 \tau}
\frac{\mathfrak{m}_{k+1,1}}{\mathfrak{m}}\, d\tau.
\label{J11}
\end{align}
The function
\begin{align*}
\int\limits_0^{R_\ell \sin(\frac{\pi}{n})}\frac{e^{- 2\tau x }}{2 i R_\ell - 2 \tau} d\tau
\end{align*}
is uniformly bounded and converges to zero as $R_\ell\to\infty$ for all $x\in(0,1]$. The measure $d\psi_0(x)$ has no atom at zero (this follows from $\psi_0(0+)=0$),
and the first term in (\ref{J11}) tends to zero by the Lebesgue theorem. The second term is evidently $O(\frac 1{R_\ell})$.

For $J_{12}$ we again use the Ces\'aro method:
\begin{align*}
({\cal C},1)\,\text{-}\lim\limits_{\ell\rightarrow\infty} J_{12}(\ell)=
\lim\limits_{\ell\rightarrow\infty}\int\limits_0^1 \frac{e^{2 i R_0x}}{2\ell}\frac{1-e^{4\pi i \ell x}}{1-e^{4\pi i x}}\frac{\mathfrak{m}_{k+1,1}}{\mathfrak{m}} \,d\psi_0(x).
\end{align*}
The integrand here is uniformly bounded and converges to zero as $\ell\to\infty$ for all $x\not\in\{0, \frac 12, 1\}$. The measure $d\psi_0(x)$ has no atoms at these points,
and the limit equals zero by the Lebesgue theorem. In a similar way we have $({\cal C},1)\,\text{-}\lim\limits_{\ell\rightarrow\infty} J_{13}(\ell)=0$.
\end{proof}


For the case of non-strongly regular boundary conditions, a sequence of circles described in Lemma \ref{lemma41} is chosen in accordance to Remark~\ref{posled}. Therefore, in this case
relation (\ref{4int+1}) and Lemma \ref{lemma41} prove the assertion of Theorem~\ref{th16}.

In the case of strongly regular boundary conditions we need the second step. We differ two subcases.

Let  $\alpha_1 \neq \overline{\alpha}_2$ in the relation (\ref{asymp}). Then, by Remark \ref{posled}, we can choose a sequence of circles of radii $R_\ell^n$ as acceptable contours separating
eigenvalues for large $\ell$. We split the sequence $R_\ell$ into two parts, $R_{2\ell}$ and $R_{2\ell-1}$, and notice that every of these subsequences satisfy the assumptions of Lemma \ref{lemma41}.
Therefore, we have
\begin{align*}
& ({\cal C},1)\,\text{-}\lim\limits_{\ell\rightarrow\infty} \int\limits_{R_{2\ell} (\Gamma^1\cup\Gamma^2)}\!\!\big({\cal I}^1_{1,k+1}+{\cal I}^1_{k+1,1}+{\cal I}^1_{k, n}+{\cal I}^1_{n, k}\big)\,dz \\
=\, &({\cal C},1)\,\text{-}\lim\limits_{\ell\rightarrow\infty} \int\limits_{R_{2\ell-1} (\Gamma^1\cup\Gamma^2)}\!\!\big({\cal I}^1_{1,k+1}+{\cal I}^1_{k+1,1}+{\cal I}^1_{k, n}+{\cal I}^1_{n, k}\big)\,dz= 0,
\end{align*}
that implies the assertion in this subcase in view of trivial relation
\begin{equation}\label{polusumma}
({\cal C},1)\,\text{-}\lim\limits_{\ell\rightarrow\infty}I_\ell=
\frac 12\big(({\cal
C},1)\,\text{-}\lim\limits_{\ell\rightarrow\infty}I_{2\ell}
+({\cal
C},1)\,\text{-}\lim\limits_{\ell\rightarrow\infty}I_{2\ell-1}\big).
\end{equation}

Now let $\alpha_1 = \overline{\alpha}_2$, $\Im(\alpha_1)>0$. Then a sequence of circles described in Lemma \ref{lemma41} can be chosen as a subsequence of acceptable contours separating
eigenvalues for large $\ell$, either $\{\gamma_{2\ell}\}$ or $\{\gamma_{2\ell-1}\}$. To be definite, let these circles be $\{\gamma_{2\ell}\}$. We use the following statement.

\begin{proposition}[\cite{NG}, Lemma 4.1]\label{proposition41}
Let $\lim\limits_{\ell\to\infty}\frac{I_\ell-I_{\ell-1}}{\ell}=0$. Then
\begin{equation}\label{odd-even}
\aligned
({\cal C},1)\text{-}\lim\limits_{\ell\rightarrow\infty}I_\ell &=&& ({\cal C},1)\text{-}\lim\limits_{\ell\rightarrow\infty} I_{2\ell}-\frac 12\cdot({\cal C},1)\text{-}\lim\limits_{\ell\to \infty}
(I_{2\ell}-I_{2\ell-1})
\\
&=&& ({\cal C},1)\text{-}\lim\limits_{\ell\rightarrow\infty} I_{2\ell-1}-\frac 12\cdot({\cal C},1)\text{-}\lim_{\ell\to \infty} (I_{2\ell+1}-I_{2\ell}),
\endaligned
\end{equation}
i.e. if one of the expressions in the right-hand side of (\ref{odd-even}) converges then the sequence in the left-hand side converges.
\end{proposition}

Thus, if we consider contours $g_{2\ell}$ enclosing $(\lambda_{2\ell}({\mathfrak{q}}))^{\frac 1n}$ and $(\lambda_{2\ell})^{\frac 1n}$ and prove that
\begin{align}\label{contur-g}
({\cal C},1)\,\text{-}\lim\limits_{\ell\rightarrow\infty}\int\limits_{g_{2\ell}}\big({\cal I}^1_{1,k+1}+{\cal I}^1_{k+1,1}+{\cal I}^1_{k, n}+{\cal I}^1_{n, k}\big)\,dz=0,
\end{align}
then the statement of theorem follows from relation (\ref{4int+1}), Lemma \ref{lemma41} and Proposition \ref{proposition41}.

The relation (\ref{asymp}) and the Cauchy residue theorem allow us to choose $g_{2\ell}$ for sufficiently large $\ell$ as the unions of two arcs and two small segments:
$$
 g_{2\ell}=(\mathfrak{r}_\ell+\varepsilon)\Gamma^1 \cup [(\mathfrak{r}_\ell+\varepsilon)e^{i\frac{\pi}n},(\mathfrak{r}_\ell-\varepsilon)e^{i\frac{\pi}n}]
 \cup(\mathfrak{r}_\ell-\varepsilon)\Gamma^1 \cup [\mathfrak{r}_\ell-\varepsilon,\mathfrak{r}_\ell+\varepsilon],
$$
 where $\mathfrak{r}_\ell=\Re(\alpha_1)+2\pi\ell$ and $\varepsilon$ is arbitrary small positive given number.

Given $\varepsilon$, the Ces\`aro limits of integrals over both arcs equal zero by Lemma \ref{lemma41}. Since the segments are separated from $(\lambda_{2\ell})^{\frac 1n}$
uniformly with respect to $\varepsilon$, the absolute values of corresponding integrals does not exceed $C\varepsilon$. Since $\varepsilon$ is arbitrary small, the relation (\ref{contur-g}) follows.
This completes the proof of Theorem \ref{th16}.

\section{On the value of the coefficient $\mathfrak C$}\label{Sec4}

\begin{proof}[Proof of Theorem \ref{th13}]
To calculate the limit in (\ref{mathfrakC}) we proceed in two
steps similarly to the proof of Theorem \ref{th16}. On the first
step, we take a sequence of circles such that there is exactly one
pair of eigenvalues between each two neighboring circles.

\begin{lemma}\label{lemma51}
For arbitrary sequence $R_\ell=R_0+2\pi\ell$ such that circles of radii $R_\ell^n$ separate pairs
of eigenvalues for sufficiently large $\ell$, the following relation holds:
\begin{align*}
\lim\limits_{\ell\rightarrow\infty} &\int\limits_{R_\ell (\Gamma^1\cup\Gamma^2)}\big({\mathbb I}^{\delta}_{1,k+1}(z)+{\mathbb I}^{\delta}_{k+1,1}(z)
+{\mathbb I}^{\delta}_{k, n}(z)+{\mathbb I}^{\delta}_{n, k}(z)\big)\,dz \\
=\ &\mathfrak c \int\limits_{0}^{\infty}\frac{e^{-iR_0}\,dt}
{(t-e^{-iR_0}\xi_1)(t-e^{-iR_0}\xi_2)}.
\end{align*}
Here $\xi_1$ and $\xi_2$ are roots of the Birkhoff polynomial,
while
\begin{align}\label{frak-c}
\mathfrak c = \frac{\mathfrak{m}_{1,k+1}-\mathfrak{m}_{k+1,1}}{i \rho^{\varkappa}\mathfrak{m}}
\end{align}
(the determinants $\mathfrak{m}$, $\mathfrak{m}_{1,k+1}$, $\mathfrak{m}_{k+1,1}$ are introduced in (\ref{frak-m}) and (\ref{frak-m-k+1})).
\end{lemma}

\begin{proof}
Using formulae (\ref{mathbb-I}) we obtain, as $\ell\to\infty$,
\begin{align*}
&\int\limits_{R_\ell (\Gamma^1\cup\Gamma^2)}\big({\mathbb I}^{\delta}_{1,k+1}(z)+{\mathbb I}^{\delta}_{k+1,1}(z)
+{\mathbb I}^{\delta}_{k, n}(z)+{\mathbb I}^{\delta}_{n, k}(z)\big)\,dz \\
=&\int\limits_{R_\ell \Gamma^1}\!\!\big({\mathbb I}^{\delta}_{1,k+1}(z)+{\mathbb I}^{\delta}_{k+1,1}(z)\big)\,dz
+ \int\limits_{R_\ell \Gamma^2}\!\!\big({\mathbb I}^{\delta}_{k, n}(z)+{\mathbb I}^{\delta}_{n, k}(z)\big)\,dz +o(1).
\end{align*}

Using relations (\ref{frak-m})--(\ref{frak-num}) proved in Appendix, we rewrite the integrands explicitly:
\begin{align*}
    {\mathbb I}^{\delta}_{1,k+1}(z)+{\mathbb I}^{\delta}_{k+1,1}(z) = \frac{e^{i z} \big(\langle \mathfrak{m}_{1,k+1}\rangle-\langle\mathfrak{m}_{k+1,1}\rangle\big)}{\langle\mathfrak{m}\rangle  + \langle\mathfrak{m}_1\rangle  e^{i z} -\rho^{\varkappa}
    \langle\mathfrak{m}\rangle  e^{2 i z}}, \qquad z \in R_\ell\Gamma^1;
\end{align*}
\begin{align*}
    {\mathbb I}^{\delta}_{k,n}(z)+{\mathbb I}^{\delta}_{n,k}(z) = &\ \frac{ \rho^{k-1}e^{iz\rho^{k-1}}( \langle\mathfrak{m}_{k,n}\rangle  - \langle\mathfrak{m}_{n, k}\rangle )}{\langle\mathfrak{m}\rangle 
    - \rho^{-\varkappa}\langle\mathfrak{m}_1\rangle  e^{iz \rho^{k-1}} -\rho^{-\varkappa} \langle\mathfrak{m}\rangle  e^{2 iz \rho^{k-1}}}\\
    = &\ \frac{e^{i \tilde{z}} (\langle\mathfrak{m}_{k+1,1}\rangle 
    -\langle\mathfrak{m}_{1,k+1}\rangle )}{\langle\mathfrak{m}\rangle  + \langle\mathfrak{m}_1\rangle  e^{i\tilde{z}} -\rho^{\varkappa}\langle\mathfrak{m}\rangle  e^{2i\tilde{z}}}\,\rho^{-1} , \qquad z \in
    R_\ell\Gamma^2,
\end{align*}
here $\tilde{z} = \rho ^ {-1} z$. Thus,
\begin{align*}
   & \int\limits_{R_\ell \Gamma^1}\!\!\big({\mathbb I}^{\delta}_{1,k+1}(z)+{\mathbb I}^{\delta}_{k+1,1}(z)\big)\,dz + \int\limits_{R_\ell \Gamma^2}\!\!\big({\mathbb I}^{\delta}_{k, n}(z)+{\mathbb I}^{\delta}_{n, k}(z)\big)\,dz \\
   =& \int\limits_{R_\ell (\rho^{-1}\Gamma^2\cup\Gamma^1)}\!\!\frac{e^{i z} (\langle\mathfrak{m}_{1,k+1}\rangle -\langle\mathfrak{m}_{k+1,1}\rangle )}{\langle\mathfrak{m}\rangle  + \langle\mathfrak{m}_1\rangle  e^{i z} -\rho^{\varkappa}
    \langle\mathfrak{m}\rangle  e^{2 i z}}\,dz=:{\mathbb J}(\ell).
\end{align*}
Using the Cauchy residue theorem we replace the integral over the
arc by the integral over three segments
\begin{align*}
(R_\ell \cos(\tfrac{\pi}{n}), -R_\ell \sin(\tfrac{\pi}{n}))
\rightarrow & \ (R_\ell,-R_\ell\sin(\tfrac{\pi}{n}))\\
\rightarrow & \ (R_\ell,R_\ell\sin(\tfrac{\pi}{n}))\rightarrow
(R_\ell\cos(\tfrac{\pi}{n}),R_\ell \sin(\tfrac{\pi}{n})).
\end{align*}

\begin{figure}

\setlength{\unitlength}{1cm}
\begin{picture}(0, 10)(-1, -5)
\put(0, 0){\vector(1, 0){10}}
\put(0, 0){\vector(0,1){4}}
\qbezier(0,0)(9.238795,3.826834)(9.238795,3.826834)
\qbezier(0,0)(9.238795,-3.826834)(9.238795,-3.826834)
\qbezier(0,0)(0.000000,-4.000000)(0.000000,-4.000000)

\thicklines
\put(-0.500000,-0.250000){0}
\put(9.000000,-0.500000){$Re(z)$}
\put(-1.200000,3.561467){$Im(z)$}
\put(8.100000,2.961467){$(R, R \sin(\frac{\pi}{n}))$}
\put(8.100000,-3.061467){$(R, -R \sin(\frac{\pi}{n}))$}
\put(3.791036,3.161467){$(R \cos(\frac{\pi}{n}), R \sin(\frac{\pi}{n}))$}
\put(3.791036,-3.561467){$(R \cos(\frac{\pi}{n}), -R \sin(\frac{\pi}{n}))$}
\qbezier(8.000000,0)(8.000000,1.560723)(7.391036,3.061467)
\qbezier(8.000000,0)(8.000000,-1.560723)(7.391036,-3.061467)
{\color{red}
\qbezier(8.000000,3.061467)(8.000000,-3.061467)(8.000000,-3.061467)
\qbezier(7.391036,-3.061467)(8.000000,-3.061467)(8.000000,-3.061467)
\qbezier(7.391036,3.061467)(8.000000,3.061467)(8.000000,3.061467)}
\end{picture}
    \caption{}
    \label{fig:2}
\end{figure}

\bigskip
\medskip
Similarly to the proof of Lemma \ref{lemma41}, integrals over the
first and the third segments tend to zero, and we obtain
\begin{align*}
   {\mathbb J}(\ell) = &\ i\int\limits_{-R_{\ell} \sin{\frac{\pi}{n}}}^{R_{\ell} \sin{\frac{\pi}{n}}}
   \frac{e^{i R_0 - \tau}(\langle\mathfrak{m}_{1,k+1}\rangle -\langle\mathfrak{m}_{k+1,1}\rangle )}{\langle\mathfrak{m}\rangle 
   + \langle\mathfrak{m}_1\rangle  e^{i R_0 - \tau} -\rho^{\varkappa} \langle\mathfrak{m}\rangle  e^{2 i R_0 - 2\tau}}\,d\tau + o(1) \\
    = &\ i\int\limits_{-\infty}^{\infty} \bigg[\,\frac{e^{i R_0 - \tau}(\mathfrak{m}_{1,k+1} -\mathfrak{m}_{k+1,1})}
    {\mathfrak{m}+ \mathfrak{m}_1 e^{i R_0 - \tau} -\rho^{\varkappa} \mathfrak{m} e^{2 i R_0 - 2\tau}}
    + e^{-|\tau|}\cdot O\big(\tfrac 1{R_\ell}\big)\bigg]\,d\tau + o(1).
\end{align*}
We make the change of variable $t=e^{-\tau}$ and recall that the
denominator is the Birkhoff polynomial of $e^{i R_0 - \tau}$. This
gives
\begin{align*}
   {\mathbb J}(\ell) = \frac{\mathfrak{m}_{1,k+1}-\mathfrak{m}_{k+1,1}}{i\rho^{\varkappa}\mathfrak{m}}\int\limits_{0}^{\infty} \frac{e^{-i R_0}\,dt}
    {(t-e^{-iR_0}\xi_1)(t-e^{-iR_0}\xi_2)}  + o(1),
\end{align*}
and the statement follows.
\end{proof}

We continue the proof of Theorem \ref{th13}. For the case of
non-strongly regular boundary conditions, a sequence of circles
described in Lemma \ref{lemma51} is chosen in accordance to
Remark~\ref{posled}. Therefore, in this case relation
(\ref{mathfrakC}) and Lemma \ref{lemma51} give (\ref{th13-nolog})
after explicit integration.

In the case of strongly regular boundary conditions we need the
second step. Let $\alpha_1 \neq \overline{\alpha}_2$ in the
relation (\ref{asymp}). Then, by Remark \ref{posled}, we can
choose a sequence of circles of radii $R_\ell^n$ as acceptable
contours separating eigenvalues for large $\ell$. We split the
sequence $R_\ell$ into two parts, $R_{2\ell}=R_0+2\pi\ell$ and
$R_{2\ell-1}=R_1+2\pi\ell$, and apply Lemma \ref{lemma51} for
these subsequences. After integration we obtain (\ref{th13-log})
using relation (\ref{polusumma}).

Using the residue theorem we can see that the resulting formula is
continuous with respect to $\xi_1$ and $\xi_2$ if $\xi_1\ne\xi_2$.
This covers the subcase $\alpha_1 =
\overline{\alpha}_2\notin\mathbb R$ and completes the proof.
\end{proof}

\begin{remark}
    If the boundary conditions are almost separated (that is, $b_j = 0$ if $j < n/2$ and $a_j = 0$ if $j \ge n/2$), then it is known that $\xi_1=-\xi_2$
(see e.g. \cite{NN}). Therefore, Theorem \ref{th13} gives
$\mathfrak{C}=0$.

If the boundary conditions are quasi-periodic (that is,
$a_j=\vartheta b_j$ ($\vartheta\ne0$) and $d_j=j$ for
$j=0,1,\dots, n-1$), then
$\mathfrak{m}_{1,k+1}=\mathfrak{m}_{k+1,1}=0$. Therefore, we also
have $\mathfrak{C}=0$.

Now let $n=2$. Then in the almost separated case $\mathfrak{C}=0$
as above. Otherwise $d_0=0$, $d_1=1$, and from relation
$\mathfrak{m}_{1,k+1}=(-1)^{d_0+d_1+1}\mathfrak{m}_{k+1,1}$ we
obtain $\mathfrak{C}=0$ again. This fact is in concordance with
results of \cite{NG}.
\end{remark}

Thus, the constant $\mathfrak{C}$ vanishes for many important
classes of operators. However, the following example shows that
$\mathfrak{C}$ can be non-zero even for a self-adjoint operator
$\mathbb L$.\medskip

{\bf Example}. We consider a fourth-order operator generated by a
differential expression ${\cal L}=D^4$ and the system of boundary
conditions:
\begin{equation*}
\left\{
\begin{array}{lr}
  y(0)+y(1)=0,\\
  y'(1)=0,\\
  y''(0)=0,\\
  y'''(0)+y'''(1)=0.
\end{array}
\right.
\end{equation*}
Direct calculation gives 
$$ \mathfrak{C} = \frac{4\sqrt{5}}{9}\cdot\tan^{-1}(4 \sqrt{5}).
$$

\section{Appendix}

Following \cite{SZN}, we define $\Delta(z) = \det({\cal W}(z))$, where
\begin{align*}
    {\cal W}_{jk}(z) = P_{j-1}(i z \rho ^{k-1}) + e^{i z \rho^{k-1}}Q_{j-1}(i z \rho ^{k-1}), \qquad j, k =1, \dots, n.
\end{align*}
Then we denote by $\Delta_{\alpha, \beta}(z)$ the determinant of a matrix that coincides with ${\cal W}(z)$ except the column $\beta$ that is replaced by $\alpha$-th column from matrix \begin{align*}
    {\cal V}_{jk}(z) = e^{i z \rho^{k-1}}Q_{j-1}(i z \rho ^{k-1}), \qquad j, k =1, \dots, n.
\end{align*}






We recall that the function $\nu(w)$ is introduced in (\ref{nu}).

\begin{proposition}[see the proof of Lemma 1 in \cite{SZN}]
\label{proposition21}
Let $\alpha\neq\beta$, $\nu=\nu(w)$, $R=|z|$, $w=e^{i Arg(z)}$ (recall that $Arg(z) \in (0, \frac{2\pi}{n}$)). Then we have as $R\to\infty$
\begin{eqnarray*}
\frac{\Delta_{\alpha,\beta}(R w)}{\Delta(R w)} &=& e^{i R w\rho^{\alpha-1}}\cdot\frac{\hat\Delta_{\alpha\beta}(Rw)}{\hat\Delta(Rw) }\, (1+O(e^{iRw\rho})) \quad\textit{if}\quad\alpha,\beta \leq \nu,\\
\frac{\Delta_{\alpha,\beta}(R w)}{\Delta(R w)} &=& e^{i R w(\rho^{\alpha-1}-\rho^{\beta-1})}\cdot\frac{\hat\Delta_{\alpha\beta}(Rw)}{\hat\Delta(Rw) }\, (1+O(e^{iRw\rho}))
\quad\textit{if}\quad\alpha \leq \nu < \beta,\\
\frac{\Delta_{\alpha,\beta}(R w)}{\Delta(R w)} &=& e^{-i R w\rho^{\beta-1}}\cdot\frac{\hat\Delta_{\alpha\beta}(Rw)}{\hat\Delta(Rw) }\, (1+O(e^{iRw\rho})) \quad\textit{if}\quad\alpha,\beta > \nu,\\
\frac{\Delta_{\alpha,\beta}(R w)}{\Delta(R w)} &=& \frac{\hat\Delta_{\alpha\beta}(Rw)}{\hat\Delta(Rw) }\, (1+O(e^{iRw\rho}))\quad \textit{if}\quad\alpha > \nu \geq \beta.\\
\end{eqnarray*}
Here $\hat{\Delta}_{\alpha,\beta}(z)$ and $\hat{\Delta}(z)$ are determinants depending on boundary conditions and satisfying the following properties:
\begin{enumerate}
 \item For large $|z|$, $\hat{\Delta}_{\alpha,\beta}(z)$ are bounded uniformly;

 \item $\hat{\Delta}(R(w)w)$ is separated from zero on arbitrary acceptable sequence of contours $\gamma_\ell$ uniformly w.r.t. $\ell$.
\end{enumerate}
\end{proposition}

The properties listed above are sufficient to prove Lemma \ref{lemma31} and Theorem \ref{th15}. In the proofs of Theorem \ref{th16} and Theorem \ref{th13}, we need explicit formulae for $\hat{\Delta}(z)$ and for those $\hat{\Delta}_{\alpha,\beta}(z)$ which are not covered by Lemma~\ref{lemma31}. 

Let $n=2k$ and recall that $w=e^{iArg(z)}$. We have 
\begin{equation}
\label{func-det}
\hat{\Delta}(z)=\det\big[{\cal A}_1|{\cal B}_1\big], \quad w\in \Gamma^1; \qquad
\hat{\Delta}(z)=\det\big[{\cal A}_2|{\cal B}_2\big], \quad w\in \Gamma^2,
\end{equation}
where
\begin{eqnarray*}
{\cal A}_1 &=&
\begin{bmatrix}
    \rho^{0\cdot d_0}\big(\langle a_{0}\rangle +\langle b_{0}\rangle e^{iz}\big) & \rho^{1\cdot d_0}\langle a_{0}\rangle  &  \vdots & \rho^{(k-1)d_0}\langle a_{0}\rangle   \\
    \vdots & \vdots & \vdots & \vdots \\
    \rho^{0\cdot d_{n-1}}\big(\langle a_{n-1}\rangle +\langle b_{n-1}\rangle e^{iz}\big) &  \rho^{1\cdot d_{n-1}}\langle a_{n-1}\rangle  & \vdots & \rho^{(k-1)d_{n-1}}\langle a_{n-1}\rangle  
\end{bmatrix},\\
\\
{\cal B}_1 &=&
\begin{bmatrix}
    \rho^{k d_0}\big(\langle a_0\rangle  e^{iz} + \langle b_{0}\rangle \big) & \rho^{(k+1) d_0}\langle b_{0}\rangle  & \vdots & \rho^{(n-1)d_0}\langle b_{0}\rangle \\
     \vdots &  \vdots & \vdots & \vdots \\
     \rho^{k d_{n-1}}\big(\langle a_{n-1}\rangle  e^{iz} + \langle b_{n-1}\rangle \big) & \rho^{(k+1) d_{n-1}}\langle b_{n-1}\rangle  & \vdots & \rho^{(n-1) d_{n-1}}\langle b_{n-1}\rangle  
\end{bmatrix};
\end{eqnarray*}
\begin{eqnarray*}
{\cal A}_2 \!\!\!&=&\!\!\!\!
\begin{bmatrix}
    \rho^{0
    }\langle a_{0}\rangle  & \vdots & \rho^{(k-2) d_0}\langle a_{0}\rangle  & \rho^{(k-1) d_0}\big(\langle a_{0}\rangle +\langle b_{0}\rangle e^{iz \rho^{k-1}}\big) \\
    \vdots & \vdots & \vdots & \vdots \\
    \rho^{0
    }\langle a_{n-1}\rangle  & \vdots & \rho^{(k-2) d_{n-1}}\langle a_{n-1}\rangle  & \rho^{(k-1) d_{n-1}}\big(\langle a_{n-1}\rangle +\langle b_{n-1}\rangle e^{iz \rho^{k-1}}\big) 
\end{bmatrix},\\
\\
{\cal B}_2 \!\!\!&=&\!\!\!\!
\begin{bmatrix}
    \rho^{k d_0}\langle b_{0}\rangle  & \vdots & \rho^{(n-2) d_0}\langle b_{0}\rangle  & \rho^{(n-1)d_0}\big(\langle a_0\rangle  e^{iz \rho^{k-1}} + \langle b_{0}\rangle \big)\\
    \vdots & \vdots & \vdots & \vdots \\
    \rho^{k d_{n-1}}\langle b_{n-1}\rangle  & \vdots & \rho^{(n-2) d_{n-1}}\langle b_{n-1}\rangle  & \rho^{(n-1) d_{n-1}}\big(\langle a_{n-1}\rangle e^{iz \rho^{k-1}}+\langle b_{n-1}\rangle \big) 
\end{bmatrix}
\end{eqnarray*}
(we recall that $\langle a\rangle $ stands for a polynomial of $z^{-1}$ with the constant term $a$).

Next, for $w \in \Gamma^1$ we have
\begin{equation}
\label{func-det-k+1-1}
\hat{\Delta}_{1,k+1}(z)=\det\big[{\cal A}_{1,k+1}|{\cal B}_{1,k+1}\big]; \qquad
\hat{\Delta}_{k+1,1}(z)=\det\big[{\cal A}_{k+1,1}|{\cal B}_{k+1,1}\big],
\end{equation}
where
\begin{eqnarray*}
{\cal A}_{1,k+1} &=&
\begin{bmatrix}
    \rho^{0\cdot d_0}\langle a_{0}\rangle  &  \vdots & \rho^{(k-1) d_0}\langle a_{0}\rangle  \\
    \vdots & \vdots & \vdots  \\
    \rho^{0\cdot d_{n-1}}\langle a_{n-1}\rangle  &  \vdots & \rho^{(k-1)d_{n-1}}\langle a_{n-1}\rangle  
\end{bmatrix};\\
{\cal B}_{1,k+1} &=&
\begin{bmatrix}
    \rho^{0 \cdot d_0}\langle b_{0}\rangle  & \rho^{(k+1) d_0}\langle b_{0}\rangle  & \vdots & \rho^{(n-1) d_0}\langle b_{0}\rangle \\
    \vdots & \vdots & \vdots & \vdots \\
    \rho^{0 \cdot d_{n-1}}\langle b_{n-1}\rangle  & \rho^{(k+1) d_{n-1}}\langle b_{n-1}\rangle  & \vdots & \rho^{(n-1)d_{n-1}}\langle b_{n-1}\rangle 
\end{bmatrix}; \\
\end{eqnarray*}
\begin{eqnarray*}
{\cal A}_{k+1,1} &=&
\begin{bmatrix}
    \rho^{k d_0}\langle b_{0}\rangle  & \rho^{1\cdot d_0}\langle a_{0}\rangle  & \vdots & \rho^{(k-1)d_0}\langle a_{0}\rangle \\
    \vdots & \vdots & \vdots & \vdots \\
    \rho^{k d_{n-1}}\langle b_{n-1}\rangle  & \rho^{1\cdot d_{n-1}}\langle a_{n-1}\rangle  & \vdots & \rho^{(k-1)d_{n-1}}\langle a_{n-1}\rangle 
\end{bmatrix};\\
{\cal B}_{k+1,1} &=&
\begin{bmatrix}
    \rho^{k d_0}\langle a_{0}\rangle  & \rho^{(k+1) d_0}\langle b_{0}\rangle  & \vdots & \rho^{(n-1)d_0}\langle b_{0}\rangle \\
    \vdots & \vdots & \vdots & \vdots \\
    \rho^{k d_{n-1}}\langle a_{n-1}\rangle  & \rho^{(k+1) d_{n-1}}\langle b_{n-1}\rangle  & \vdots & \rho^{(n-1)d_{n-1}}\langle b_{n-1}\rangle 
\end{bmatrix}.\\
\end{eqnarray*}

Similarly, for $w \in \Gamma^2$ we have
\begin{equation}
\label{func-det-n-k}
\hat{\Delta}_{k,n}(z)=\det\big[{\cal A}_{k,n}|{\cal B}_{k,n}\big]; \qquad
\hat{\Delta}_{n,k}(z)=\det\big[{\cal A}_{n,k}|{\cal B}_{n,k}\big],
\end{equation}
where
\begin{eqnarray*}
{\cal A}_{k,n} &=&
\begin{bmatrix}
\rho^{0\cdot d_0}\langle a_{0}\rangle  &  \vdots & \rho^{(k-1)d_0}\langle a_{0}\rangle  \\
    \vdots & \vdots & \vdots \\
    \rho^{0\cdot d_{n-1}}\langle a_{n-1}\rangle  &  \vdots & \rho^{(k-1)d_{n-1}}\langle a_{n-1}\rangle  \\
\end{bmatrix};\\
{\cal B}_{k,n} &=&
\begin{bmatrix}
\rho^{k d_0}\langle b_{0}\rangle  & \vdots & \rho^{(n-2)d_0}\langle b_{0}\rangle  & \rho^{(k-1)d_0}\langle b_{0}\rangle \\
 \vdots & \vdots & \vdots & \vdots \\
\rho^{k d_{n-1}}\langle b_{n-1}\rangle  & \vdots & \rho^{(n-2)d_{n-1}}\langle b_{n-1}\rangle  & \rho^{(k-1)d_{n-1}}\langle b_{n-1}\rangle  \\
\end{bmatrix}; \\
\end{eqnarray*}
\begin{eqnarray*}
{\cal A}_{n,k} &=&
\begin{bmatrix}
    \rho^{0 \cdot d_0}\langle a_{0}\rangle  & \vdots  & \rho^{(k-2) d_0}\langle a_{0}\rangle & \rho^{(n-1)d_0}\langle b_{0}\rangle \\
    \vdots & \vdots & \vdots & \vdots\\
    \rho^{0 \cdot d_{n-1}}\langle a_{n-1}\rangle  &  \vdots & \rho^{(k-2) d_{n-1}}\langle a_{n-1}\rangle  & \rho^{(n-1)d_{n-1}}\langle b_{n-1}\rangle \\
\end{bmatrix};\\
{\cal B}_{n,k} &=&
\begin{bmatrix}
    \rho^{k d_0}\langle b_{0}\rangle  & \vdots & \rho^{(n-2)d_0}\langle b_{0}\rangle  & \rho^{(n-1)d_0}\langle a_{0}\rangle \\
    \vdots & \vdots & \vdots & \vdots\\
    \rho^{k d_{n-1}}\langle b_{n-1}\rangle  & \vdots & \rho^{(n-2)d_{n-1}}\langle b_{n-1}\rangle  &
    \rho^{(n-1)d_{n-1}}\langle a_{n-1}\rangle  \\
\end{bmatrix}.\\
\end{eqnarray*}

We expand $\hat{\Delta}(z)$ in exponentials and obtain
 \begin{align}\label{frak-m}
 \hat{\Delta}(z) &=\langle \mathfrak{m}_2\rangle  e^{2 i z} + \langle \mathfrak{m}_1\rangle  e^{i z} + \langle \mathfrak{m}\rangle , &
 w \in \Gamma^1;\\
 \hat{\Delta}(z)&=\langle \mathfrak{m}'_2\rangle  e^{2 iz \rho^{k-1}} + \langle \mathfrak{m}'_1\rangle  e^{iz \rho^{k-1}} + \langle \mathfrak{m}\rangle ,
 &
 w \in \Gamma^2
 \nonumber
 \end{align}
(the constant terms evidently coincide and do not vanish by regularity of boundary conditions (\ref{Boundary_Conditions})).

Taking the common multiplier $\rho^{d_j}$ over from the $j$-th row in the determinant $\mathfrak{m}_2$ (see the proof of \cite[Theorem 1.1]{Naz09}) we obtain
$\mathfrak{m}_2=-\rho^\varkappa \mathfrak{m}$. In the same way, $\mathfrak{m}'_2=-\rho^{-\varkappa} \mathfrak{m}$ and
$\mathfrak{m}'_1=-\rho^{-\varkappa} \mathfrak{m}_1$.

Finally, we have
\begin{align}\label{frak-m-k+1}
 \hat{\Delta}_{1,k+1}(z) = \langle \mathfrak{m}_{1,k+1}\rangle ,& \qquad \hat{\Delta}_{k+1,1}(z) = \langle \mathfrak{m}_{k+1,1}\rangle ,
 & w \in \Gamma^1;\\
 \hat{\Delta}_{k,n}(z) = \langle \mathfrak{m}_{k,n}\rangle ,& \qquad
 \hat{\Delta}_{n,k}(z) = \langle \mathfrak{m}_{n,k}\rangle , & w \in \Gamma^2,
 \nonumber
 \end{align}
and a similar calculation gives 
\begin{align}\label{frak-num}
\mathfrak{m}_{k,n}=-\rho^{-\varkappa}\mathfrak{m}_{k+1,1}, \qquad \mathfrak{m}_{n,k}=-\rho^{-\varkappa}\mathfrak{m}_{1,k+1}.
\end{align}

\subsection*{Acknowledgements}
We are grateful to A.V. Badanin for valuable comments.
This work was supported by the Russian Science
Foundation, grant no. 17-11-01003.


\begin{thebibliography}{XX}

\bibitem{DM}
P.~Djakov, B.~Mityagin, {\em Trace formula and spectral Riemann surfaces for a class of tridiagonal matrices}, J. Approx. Theory {\bf 139} (2006), 293--326.

\bibitem{Gal19}
E.D.~Galkovskii, {\em A trace formula for a high-order differential operator on a segment with the last coefficient perturbed by a finite signed measure}, Funct. Analysis and Its Appl. {\bf 53} (2019), N2, 64--67 (Russian).

\bibitem{NG}
E.D.~Galkovskii, A.I.~Nazarov, {\em A general trace formula for a regular differential operator on a segment with the last coefficient perturbed by a finite signed
measure}, Algebra \& Analysis {\bf 30} (2018), N3, 30--54 (Russian).

\bibitem{GL}
I.M.~Gelfand, B.M.~Levitan, {\em On a simple identity for the eigenvalues of a second-order differential operator}, DAN SSSR, {\bf 88} (1953), 593--596 (Russian).

\bibitem{Konechnaya}
N.N.~Konechnaya, T.~A.Safonova, R.N.~Tagirova, {\em Asymptotics of eigenvalues and regularised first-order trace of the Sturm-Liouville operator with
$\delta$-potential}, Vestnik SAFU, 2016, N1, 104--113 (Russian).

\bibitem{Naimark}
M.A.~Naimark, {\em Linear differential operators}, ed.~2, Moscow, Nauka, 1969 (Russian). English transl. of the 1st ed.:
{\em Linear Differential Operators, V.1: Elementary theory of linear differential operators}, Harrap, 1967.

\bibitem{Naz09}
A.I.~Nazarov, {\em Exact $L_2$-Small Ball Asymptotics of Gaussian Processes and the Spectrum of Boundary-Value Problems}, J. Theor. Probab. {\bf22} (2009), N3, 640--665.

\bibitem{NN}
A.I.~Nazarov, Ya.Yu.~Nikitin, {\em Exact $L_2$-small ball behavior of integrated Gaussian processes and spectral asymptotics of boundary value problems}, Prob. Th. Rel. Fields, {\bf 129} (2004), N4, 469--494.

\bibitem{SZN}
A.I.~Nazarov, D.M.~Stolyarov, P.B.~Zatitskiy, {\em Tamarkin equiconvergence theorem and trace formula revisited}, J. Spectral Theory, {\bf 4} (2014), N2, 365--389.

\bibitem{SPSurvey}
V.A.~Sadovnichii, V.E.~Podolskii, {\em Traces of operators}, 
Russian Math. Surveys, {\bf 61} (2006), N5, 885--953.

\bibitem{S2000}
A.M.~Savchuk, {\em First-order regularised trace of the Sturm-Liouville operator with $\delta$-potential}, 
Russian Math. Surveys, {\bf 55} (2000), N6, 1168--1169.

\bibitem{SSh}
A.M.~Savchuk, A.A.~Shkalikov, {\em Trace Formula for Sturm-Liouville Operators with Singular Potentials},
Math. Notes, {\bf 69} (2001), N3, 387--400.

\bibitem{Shv}
R.F.~Shevchenko, {\em On the trace of a differential operator}, 
Soviet Math. Dokl.,
{\bf 6} (1965), 1183--1186.

\bibitem{Shk}
A.A.~Shkalikov, {\em Boundary-value problems for ordinary differential equations with a parameter in the boundary conditions},
Trudy Seminara im. I.G. Petrovskogo, {\bf 9} (1983), 190--229 (Russian). English transl.: J. Soviet Math. {\bf 33} (1986), 1311--1342.

\bibitem{VinSad02}
V.A.~Vinokurov, V.A.~Sadovnichii, {\em The Asymptotics of Eigenvalues and Eigenfunctions and a Trace Formula for a Potential with Delta Functions}, 
Diff. Eqs. {\bf 38} (2002), N6, 772--789.

\end{thebibliography}
\end{document}